\documentclass[12pt]{amsart}
\newtheorem{thm}{Theorem}
\newtheorem{lem}[thm]{Lemma}
\newtheorem{de}[thm]{Definition}
\newtheorem{prop}[thm]{Proposition}

\def\N{{\Bbb N^+}}

\def\nt{\noindent}

\def\ph{\varphi}

\def\be{\begin{enumerate}}
\def\ee{\end{enumerate}}
\def\ar{\rightarrow}
\def\to{\rightarrow}

\def\er{\exists^R}
\def\fr{\forall^R}
\def\ert{\exists^{R,2}}
\def\frt{\forall^{R,2}}
\def\en{\exists^N}
\def\fn{\forall^N}
\def\en{\exists^{\N}}
\def\fn{\forall^{\N}}

\usepackage{mathrsfs}

\begin{document}
\vskip5cm
\nt {\large\bf Interpreting Intuitionistic Analysis in Intuitionistic Real Algebra}

\vskip1cm

\nt{\bf Mikl\'os Erd\'elyi-Szab\'o}
\vskip0.5cm

\nt R\'enyi Alfr\'ed Institute of Mathematics,

\nt 1053 Budapest, Re\'altanoda u. 13-15.

\nt Hungary\footnote{e-mail: mszabo@@renyi.hu}
\vskip1cm

\nt {\bf Abstract.}\quad We show that in the presence of a strengthened Kripke's schema --- a plausible addition to the axiomatisation of intuitionistic analysis (see in e.g. \cite{sco1} or \cite{sco2}) --- choice sequences can be recursively encoded in intuitionistic real algebra. Choice sequences are intuitionistically meaningful counterparts of sequences of natural numbers, and with them intuitionistic analysis can be interpreted in intuitionistic real algebra. 
\vskip0.6cm

\nt {\bf Mathematics Subject Classification:} 03-B20, 03-F25, 03-F35, 03-F60.
\vskip0.6cm

\nt {\bf Keywords:} Intuitionism, Kripke's Schema, Second-order Heyting arithmetic, Intuitionistic Real algebra, Interpretation.
\pagestyle{myheadings}
\markboth{Mikl\'os Erd\'elyi-Szab\'o}{Encoding Sequences in Intuitionistic Real Algebra} 
\vskip1cm

\section{Introduction}
\noindent The present article is part of a series of papers on the complexity and expressive power of intutitonistic real algebra, the intuitionistic version of the theory of real closed ordered fields. 
The models of these theories  are the real algebraic parts of models of intuitionistic analysis based on complete Heyting algebras --- complete, bounded lattices with a binary operation $a\to b$ such that $c\wedge a\leq b$ if and only if $c\leq a\to b$. Heyting algebras correspond to intuitionistic  propositional logic in the same way Boolean algebras to classical propositional logic. In particular they serve as the structure of truth values: if a propositional sentence is an intuitionistic tautology, then it evaluates to the top element in every Heyting algebra. If the algebra in question is complete, this interpretation can be extended to intuitionistic predicate calculus, moreover there is a standard way to interpret choice sequences as special maps from the set of pairs of natural numbers to the algebra, $\xi(m,n)$ being the truth value of the statement that the $m$-th element of the sequence is $n$. For more about choice sequences the interested reader may consult with eg. \cite{tr1}.  

The simplest model of intuitionistic real algebra based on a Heyting algebra (in this case the open-set algebra of Baire space) is due to D. Scott. In \cite{scott} he defined a model where the elements are continuous functions from Baire space to the structure of reals. Moschovakis in \cite{mosc} defined a model for analysis where the structure of reals, via a standard representation due to Vesley (and described below) agrees with Scott's interpretation.  Other models of this type were investigated for example by M.D. Krol in \cite{krol} and by P. Scowcroft in \cite{sco1} and in \cite{sco2}.

\bigskip
Since classical real algebra is a particularly nice theory, decidable with complete axiomatisation, it was hoped that the intuitionistic models may also have nice properties. In the spirit of the Feferman-Vaught theorem one might have hoped that a reduction of the real-algebraic theory of Scott's model to some combination of the decidable classical theory of reals and some decidable topological structure is possible.  The first results were indeed about decidability. In \cite{scott2} Scott proved that the universal fragment of the order structure of his model is decidable. In the papers \cite{sco3} and \cite{sco4} Scowcroft showed that the validity in the model of certain $\forall\exists$-sentences of the language $L$ of ordered rings is also decidable. We proved the decidability of Scott's model restricted to the language of ordered $\mathbb Q$-vectorspaces in \cite{eszmjsl1} and that proof indeed used, via Rabin's theorem, the decidability of the underlying topological structure. But the full real algebraic structure turned out to be undecidable. There were results by Smory\'{n}ski (in \cite{sm}) and by Cherlin (in \cite{ch}) suggesting that this might be the case. Indeed, building on Cherlin's work, in \cite{eszmmlq} we were able to interpret true first order arithmetic in the $L$-theory of Scott's model. Then, in \cite{eszmjsl} we extended this result to the classes of models defined in Scowcroft's papers \cite{sco1} and \cite{sco2}  and in Krol's paper \cite{krol}. In \cite{eszmjsl} we interpreted true second-order arithmetic in the $L$-theory of Scott's model, and also showed that Moschovakis' model  for analysis can be encoded in true second-order arithmetic giving the exact complexity of the $L$-theory of Scott's model.

\bigskip
In \cite{eszmmlq2} we have shown an encoding of true second-order arithmetic in models, in this paper we give an encoding using an axiom system.  The point is that the coding will be direct, no Gödel numbering of syntax is needed. The effective translation of intuitionistic analysis into real algebra presented here yields an intuitionistically meaningful undecidability result; the previous arguments via models is less compelling if they are treated in a classical metamathematics.\footnote{observation by Philip Scowcroft} The result parallels the results in \cite{eszmmlq} and \cite{eszmjsl}  where we moved from proving the interpretability of the natural number structure in a model of intuitionistic real algebra to proving interpretability from an axiom system, moving away from the peculiarities of a given structure. A similar attempt was made in \cite{eszmarch1} where a relativised generalisation of Random Kripke's Schema (see below) was proposed and it was shown that the interpretability of second-order Heyting arithmetic in intuitionistic real algebra follows from the axiom system $\mathscr{T}$ from \cite{sco1} with Relativised Random Kripke's Schema (RR-KS). Unfortunately the consistency of the full version of RR-KS  with $\mathscr{T}$  remains open. 
In this paper we sidestep the problem of consistency to get a related result - the interpretability of full second-order arithmetic in   intuitionistic real algebra using the axiom system $\mathscr{T}$  with Random Kripke's Schema. In fact, if the first order structure of natural numbers is definable in an intuitionistic real algebra (as in the real algebraic structures of the models described in \cite{sco1} and in \cite{sco2} - the definability there follows from R-KS and the usual axioms, for the details see \cite{eszmjsl}), then so is intuitionistic analysis. 
\bigskip

From \cite{eszmjsl} let us recall the following.
We shall use the language $L_1$ from \cite{sco1}.
It contains 
two sorts of variables -- $m$, $n$, $k$, etc. ranging over the elements of 
$\omega$, and $\alpha$, $\beta$, etc. ranging over choice sequences.
We also have the equality symbol $=$.
 It 
will be used in atomic formulas of the form 
$t=t'$ or $\xi(t)=t'$ where $t$ and $t'$ are terms of natural-number sort and 
$\xi$ is of choice sequence sort.

\smallskip

In the Appendix of \cite{eszmjsl} we defined Random Kripke's Schema as the following axiom schema of analysis:

\[ \hbox{R-KS}(\ph)\equiv \exists\beta[(\exists n(\beta(n)>0)\ar\ph)\wedge 
(\neg\exists n(\beta(n)>0)\ar\neg\ph)\wedge \]
\[\forall k>0(\neg\exists n(\beta(n)=k)\ar\ph\vee\neg\ph)\wedge 
\forall k>0\,\forall n((\beta(n)=k)\ar\forall m\geq n(\beta(m)=k))]\]

\nt where $\ph$ is a formula that does not contain the choice sequence variable $\beta$ free. 

\medskip

%
%

\nt Also in \cite{eszmjsl} we have proved the following.

\begin{enumerate}
\item[(i)] The models of intuitionistic analysis described in \cite{sco1} and in \cite{sco2} are models of R-KS (\cite{eszmjsl}, Appendix). 
\item[(ii)] From a standard axiom system augmented with R-KS the definability in the language $LOR$  of ordered rings of the set of natural numbers follows: there is a formula  $N(x)$ with one free variable in the language of ordered rings such that  $\exists k\in\N (x=k)\equiv N(x)$ follows from the axioms in two-sorted intuitionistic predicate calculus with equality (\cite{eszmjsl}, p1015, Theorem 1). 

\end{enumerate}

Using these results one may prove the following.

\begin{prop}\label{rec}
The ordered semiring structure of natural numbers has a uniform definition (ie. the defining formulas do not depend on the models in question) in the real algebraic part of any model of the axiom system $\mathscr{T}$ from \cite{sco1} (also in \cite{sco2}) with R-KS added. From this follows that each recursive function has such a uniform definition. The restriction to $\{x | N(x)\}$ of the quantifiers in the formula defining the recursive function/relation in the structure of natural numbers works.
\end{prop}\qed
\bigskip

Also, the real algebraic structure  is defined in intuitionistic analysis:

In~\cite[pages 134-135]{kv} Vesley considers a species $R$ of {\it real-number generators}: $\xi\in R$ (also denoted by $R(\xi)$)
if and only if the sequence $2^{-x}\xi(x)$ ($x\in\N$) of dyadic fractions is a Cauchy-sequence with
$\forall k\exists x\forall p |2^{-x}\xi(x)-2^{-x-p}\xi(x+p)|<2^{-k}$, i.e.
if and only if $$\forall k\exists x\forall p 2^k|2^p\xi(x)-\xi(x+p)|<2^{x+p} \quad(*)$$ Note that to make what comes below simpler we  changed the domains of real number generators from $\omega$ to $\N$.

\smallskip

Equality, ordering, addition and multiplication on $R$ are also defined 
 (cf. also \cite[pages 20-21]{heyting}),
and the definitions can be extended readily to polynomials  of choice sequences with nonnegative coefficients.

$\xi$ is a {\it global real-number generator} just in case $ R(\xi)$ holds. We shall use the letters  $f$, $g$, $u$ etc.  to range over global real-number generators and we shall use the defined quantifiers 
$$\er u\,\theta:\equiv  \exists u(R(u)\wedge \theta)$$
$$\fr u\,\theta:\equiv  \forall u(R(u)\ar \theta)$$

\noindent and when we have the definition $N(x)$ of the natural numbers
$$\en u\,\theta:\equiv  \exists u(N(u)\wedge \theta)$$
$$\fn u\,\theta:\equiv  \forall u(N(u)\ar \theta)$$

For each natural number $n$ there is a corresponding global real-number generator $f_n$ (also denoted by $n$ in the context of real numbers only) defined as follows: 
$f_n(l) =n2^l$. Then in the real numbers,
\[  f_n f=\overbrace{f+\cdots+f}^{n} \]


A choice sequence $\eta$ with range in the set $\{0,1\}$ (a $0-1$ sequence) corresponds to the real number generator 
$$\xi: x\geq 1\longmapsto \Sigma_{i=1}^x\;\eta(i) 2^{x-i}$$
with the corresponding real
$$\Sigma_{i=1}^\infty\;\eta(i) 2^{-i} = \lim_{x\to\infty}  \Sigma_{i=1}^x\;\eta(i) 2^{x-i} $$ in the closed unit interval.
In fact, $$2^k|2^p\xi(x)-\xi(x+p)| = 2^k(\Sigma_{i=1}^{x+p}\;\eta(i) 2^{x+p-i} - 2^p\;\Sigma_{i=1}^x\;\eta(i) 2^{x-i}) \leq $$
$$2^k\;\Sigma_{i=x+1}^{x+p}\; 2^{x+p-i} < 2^{k+p}$$
so $x=k$ in $(*)$ shows that $\xi$ is a real number generator.
\medskip

\noindent We shall use the (definable) quantification $\ert$ and $\frt$ over real number generators with range in the set $\{0,1\}$.

\begin{de}
A real number $y$ is $p$-rational ($p\in\N$) if $\en q\; (y\cdot 2^p = q)$. Note that "y is p-rational" (for some $p$) and "y is not p-rational for any p" are definable in the language $LOR$.
\end{de}

\section{Encoding choice sequences}

We encode choice sequences as real numbers in the closed unit interval such a way that atomic formulas of the form $\beta(m)=k$ could be translated into a formula $C(b,m,k)$ in the language of ordered rings where $b$ is a code for $\beta$. Also the set of codes will be definable, a code will be unique (with respect to equality of real numbers) and each code will be a code of a unique choice sequence (see Proposition~\ref{cd2} below).

Some notation: $\langle.,.\rangle$ will denote a recursive injective and surjective pairing function.

The code $C_\xi$ of a choice sequence $\xi$ will be a real number with generator a $0-1$ sequence $C'_\xi$ such that
\be
\item[(i)]  $C'_\xi(2\langle m,k\rangle) = 1$ iff $\xi(m)=k$  ($m,k\in\omega$).
\item[(ii)]  $C'_\xi(4m+1) = 0$ and $C'_\xi(4m+3) = 1$  ($m\in\omega$).
\ee

\smallskip

The second (technical) condition above is needed to make sure that any real number is a code of at most one sequence and to be able to deduce properties of a $0-1$ real number generator from the properties of the corresponding real (see Lemma \ref{cd1}). 
\medskip

Now some details. First of all we have to express $C'_\xi(p) = 0$ with a formula on the languge $LOR$ of ordered rings involving the real numbers $C_\xi\in[0,1]$ and $p$ with $N(p)$. The next lemma gives the required translation.

 In what follows, for every $0-1$ sequence $x'$ let $x$ be the corresponding real number $x = \Sigma_{i=1}^\infty\;x'(i) 2^{-i}.$ Then 
$x'(p) = 0$ means that in the dyadic expansion of  $x$ there is no $1/2^p$. The next lemma connects values of $x'$ with properties of $x$ expressible by a $LOR$-formula.
\medskip

\begin{lem}\label{cd1}
\begin{enumerate}
\item[(i)]  If $x'(m) = 0$ for all $m>p$,  then $\en q\; (x\cdot 2^p = q)$, ie. x is p-rational; vice versa, if x is p-rational ($\en q\; (x\cdot 2^p = q)$) and $x'(m) = 0$ for some $m>p$, then $x'(m)=0$ for all $m>p$. If $x'(4m+1) = 0$ and $x'(4m+3) = 1$  for all $m\in\omega$, then for all $p\in\N$ $x$ is not $p$-rational.
\item[(ii)] Assume $x'(4m+1) = 0$ and $x'(4m+3) = 1$  for all $m\in\omega$ and $x'(p)=0$. Let $x_p = \Sigma_{i=1}^{p-1}x'(i)\cdot 1/{2^i}$ a $(p-1)$-rational real. Then 
$\en q\; (x_p\cdot 2^{p-1} = q)$ and $0< x-x_p<1/{2^p}$.
\item[(iii)] Assume that $x'(i) = 0$ for some $i>p$. If $\exists y\;\en q\; (y\cdot 2^{p-1} = q \wedge 0< x-y<1/{2^p})$,  then $x'(p) = 0$. In this case $y$ has a representation $y'$ such that $y'(j)=x'(j)$ for all $j<p$.

\end{enumerate}
\end{lem}

\begin{proof}
\begin{enumerate}
\item[(i)] is trivial.\qed
\item[(ii)] $x-x_p = \Sigma_{i\geq p}x'(i)\cdot 1/{2^i} \geq 1/{2^{4p+3}}>0$.

\noindent Since $x'(p) = 0$ and $x'(4p+1)=0$, $x-x_p = \Sigma_{i > p}x'(i)\cdot 1/{2^i} < \Sigma _{i > p} 1/{2^i} = 1/2^p$. \qed
\item[(iii)] 
First note that $x'$ being a $0-1$ choice sequence implies that $x'(k) = 0$ or $x'(k) = 1$ for all $k$.
 By $(i)$ $y$ is of the form   $y  = \Sigma_{i=1}^{p-1}y'(i)\cdot 1/{2^i}$ for an appropriate representation $y'$.  

If $j<p$ is the least index with  $y'(j) \neq x'(j)$ then if $y'(j) < x'(j)$ we have $x-y\geq 1/2^j - \Sigma_{m=j+1}^{p-1} 1/2^m >  1/2^p$. If $y'(j) > x'(j)$ then  $x-y\leq -1/2^j + \Sigma_{m>j} x'(m)\cdot 1/2^m <  -1/2^j +  1/2^j = 0$, since $x'(i) = 0$ for some $i>p$.

If for all $i<p$ $y'(i) = x'(i)$ and $x'(p) = 1$, then $x-y \geq 1/2^p$. 

\end{enumerate}
\end{proof}


To apply the lemma we need to take care of the condition on the "returning zeros". For a $0-1$  sequence $x'$ corresponding to the real number $x$ let $[x,p]$ denote the $LOR$-formula equivalent to 
$$\exists y (\en q\; (y\cdot 2^{p-1} = q) \wedge 0< x-y<1/{2^p})$$

\begin{lem}\label{cd!}
$\forall m\; x'(4m+1) = 0$,  $x'(4m+3) = 1$  and $x'(p) = 0$  iff  $\forall m ([x, 4m+1]\wedge\neg[x, 4m+3]) \wedge [x, p]$.
\end{lem}

\begin{proof}
First assume that $\forall m\; x'(4m+1) = 0$,  $x'(4m+3) = 1$  and $x'(p) = 0$.  Then by Lemma \ref{cd1} (ii)  $\forall m [x, 4m+1]$ and $[x, p]$ follows. To finish the left to right direction assume that $[x, 4m+3]$ for some $m\in\omega$. Then there is $y$, $(4m+2)$-rational such that
$0< x-y<1/{2^{4m+3}}$, ie. since $x'(4m+3) = 1$,

\noindent $0< 2^{4m+2}\Sigma_{i=1}^{4m+2}(x'(i)\cdot 1/{2^i})- 2^{4m+2}\cdot y + 1/2 + 2^{4m+2}\cdot\Sigma_{i=4m+4}^{\infty}(x'(i)\cdot 1/{2^i}) <1/2$.

\noindent $A=2^{4m+2}\Sigma_{i=1}^{4m+2}(x'(i)\cdot 1/{2^i})- 2^{4m+2}\cdot y$ is an integer, and since $\forall m\; x'(4m+1) = 0$,  $x'(4m+3) = 1$,  for $q=2^{4m+2}\cdot\Sigma_{i=4m+4}^{\infty}(x'(i)\cdot 1/{2^i})$ we have $0<q<1/2.$ 

\noindent Thus $0< A + 1/2 + q < 1/2$, ie. $-1 < -1/2-q<A<-q<0$, a contradiction for an integer $A$. So $\neg[x, 4m+3]$ follows.
\medskip

Next assume that $\forall m ([x, 4m+1]\wedge\neg[x, 4m+3]) \wedge [x, p]$. If $x'(i) = 0$ for some $i$, $4m+1 < i < 4m+7$,  a decidable property, then by Lemma \ref{cd1} (iii) $x'(4m+1)=0$ follows from $[x,4m+1]$. Otherwise $x'(i)=1$ for every $4m + 1 < i < 4m+7$. By $[x,4m+5]$ there is a $(4m+4)$-rational $y$ such that $0<x-y<1/2^{4m+5}$ and we may assume that $y  = \Sigma_{i=1}^{4m+4}(y'(i)\cdot 1/{2^i})$. As in the proof of  Lemma \ref{cd1} (iii), 
if $j<4m+5$ is the least index with  $y'(j) \neq x'(j)$ then if $y'(j) < x'(j)$ we have $x-y\geq 1/2^j - \Sigma_{i=j+1}^{4m+4}\; 1/2^i >  1/2^{4m+5}$. If $y'(j) > x'(j)$ then  $x-y\leq -1/2^j + \Sigma_{i>j} x'(i)\cdot 1/2^i \leq  -1/2^j +  1/2^j = 0$, both cases contradict to $0<x-y<1/2^{4m+5}$.  So 
we have that $y'(j)=x'(j)$ for all $j<4m+5$. But then 
$0<x-y$ and, since $x'(i)=1$ for every $4m + 1 < i < 4m+7$,  $ x-y = \Sigma_{i=4m+5}^{\infty} x(i)/{2^i} \geq 1/2^{4m+4} - 1/2^{4m+6} > 1/2^{4m+5}$, again a contradiction. So $x'(i) = 0$ for some $i > 4m+1$ and $x'(4m+1)=0$ follows. The same reasoning proves that  $x'(p)=0$.

\noindent If $x'(4m+3)=0$ then, since $x'(4m+5)=0$, $[x,4m+3]$ follows from Lemma \ref{cd1} (ii), a contradiction, so $x'(4m+3)=1$. 
\end{proof}
\medskip

\begin{de}\label{codedef}
For a choice sequence $\xi''$ let $\xi'$ be the $0-1$ sequence with alternating $0$ and $1$ at odd natural numbers and with $\xi''(m)=k$ (decidable) iff $\xi'(2\langle m, k\rangle) = 1$ (decidable).
Let $\xi$ be the corresponding real number.   
The code of an atomic formula of the form $\xi''(k)=m$ is the $LOR$-formula $C(\xi,k,m)$ equivalent to $\forall p ([\xi, 4p+1]\wedge \neg[\xi, 4p+3]) \wedge \neg[\xi, 2\langle k, m\rangle]$.

\noindent A  $0-1$ sequence $x'$ is a code ($CODE(x')$), if it has alternating $0$ and $1$ at odd natural numbers, and for all $m$ there is a unique $k$ such that $x'(2\langle m, k\rangle) = 1$, ie. 
$$CODE(x') \hbox{ iff } \forall m ([x, 4m+1]\wedge\neg[x,4m+3]\wedge \exists !k\;\neg[x,2\langle m,k\rangle]).$$ So the set of codes is definable in real algebra by a $LOR$-formula.

For a choice sequence $\alpha$ and a $0-1$ sequence $u'$, $u'$ is the code of $\alpha$, denoted as $C(\alpha, u')$ if  $CODE(u') \;\wedge\;\forall k \forall m\; (\alpha(k) = m \leftrightarrow C(u,k,m))$ where $u$ is the real number corresponding to $u'$. 
\end{de}

From the above definitions and lemmas follows that each choice sequence has a unique code and a code corresponds to a unique choice sequence:

\begin{prop}\label{cd2}
\begin{enumerate}
\item[(i)] $\forall \alpha\ert !u'\; C(\alpha, u')$
\item[(ii)] $\frt u'(CODE(u') \to \exists! \alpha\; C(\alpha, u'))$
\ee
\end{prop}
\begin{proof}

$(i)$ Fix the sequence $\alpha$ and assume $C(\alpha, u')$ for some $0-1$ sequence  $u'$.

 \noindent By Definition~\ref{codedef}, for the real number $u$ corresponding to $u'$
$$\forall k \forall m\; (\alpha(k) = m \leftrightarrow C(u,k,m))$$ 
$$\forall k \forall m\; (\alpha(k) = m \leftrightarrow \forall p [u, 4p+1] \wedge\neg[u, 4p+3]\wedge \neg[u, 2\langle k, m\rangle])$$ and by Lemma~\ref{cd!}
$$\forall k \forall m\; (\alpha(k) = m \leftrightarrow \forall p\;( u'(4p+1) = 0 \wedge  u'(4p+3) = 1 )\wedge u'(2\langle k, m\rangle) = 1)$$ Since the pairing function is $1-1$ and onto, this determines a unique $0-1$ sequence  $u'$ fulfilling the requirement.

(ii) is similar using the equivalences of $(i)$.
\end{proof}

A formal language for intuitionistic analysis is given in \cite{kv}, 5.5; a more concise treatment is found in \cite{sco1}, pp. 149-150. Formulas are built from first order atomic formulas and formulas of the form $\xi(t_1)=t_2$ where $t_1$ and $t_2$ are terms of natural number sort, $\xi$ is of sequence short. In the following definition and proposition we restrict ourselves to formulas built from these atomic ingredients using the usual connectives and first and second-order quantifiers. However apart from functions that can be explicitly defined using the usual language of arithmetic (symbols for addition, multiplication, 0, 1 and less than) there are term-building operations used in the formal language mentioned above whose arguments include variables of both sorts: 
$$\Sigma_{y<x}\alpha(y), \quad \Pi_{y<x}\alpha(y), \quad \min_{y<x}\alpha(y), \quad \max_{y<x}\alpha(y)$$
we can exclude these using principles costumary in the axiomatic development, for example for finite sums ($S$ denotes the successor function)
$$\forall\alpha\exists\beta[\beta(0)=0 \wedge \forall z(\beta(S(z))= \beta(z)+\alpha(z))]$$
so if  a term $\Sigma_{y<x}\alpha(y)$ occurs in a formula $\ph$, $\ph(\Sigma_{y<x}\alpha(y))$ maybe replaced by
$$\exists\beta[\beta(0)=0 \wedge \forall z(\beta(S(z))= \beta(z)+\alpha(z))\wedge\ph(\beta(x))]$$
The treatment of finite products, minima and maxima is similar.
\begin{de}\label{tr} For each choice sequence variable $\xi''$ let $\xi'$ be a designated variable ranging over $0-1$ sequences and $\xi$ a designated variable ranging over reals. For each $L_1$-formula $\theta$ we define its $LOR$ translation $\tau(\theta)$ as follows.
\begin{itemize}
\item[-] If $\theta$ is first order atomic, $\tau(\theta):=\theta$
\item[-] $\tau(\xi_i''(t_1)=t_2)$ is the $LOR$-formula $C(\xi_i, t_1, t_2)$
\item[-] $\tau(\psi_1\circ \psi_2) := \tau(\psi_1)\circ \tau(\psi_2)$ for $\circ = \wedge, \vee, \ar$
\item[-] $\tau(\neg\psi_1) := \neg\tau(\psi_1)$ 
\item[-] $\tau(\exists x\psi_1)$  is the $LOR$-formula equivalent to $\en x(\tau(\psi_1))$
\item[-] $\tau(\forall x\psi_1)$  is the $LOR$-formula equivalent to  $\fn x(\tau(\psi_1))$
\item[-] $\tau(\exists \xi''\psi_1)$  is the $LOR$-formula equivalent to  $\ert \xi' \;(CODE(\xi') \wedge \tau(\psi_1)(\xi'))$
\item[-] $\tau(\forall \xi''\psi_1)$  is the $LOR$-formula equivalent to  $\frt \xi'\;(CODE(\xi') \to \tau(\psi_1)(\xi'))$

\end{itemize}
\end{de}

\medskip

\begin{prop}\label{cd!!}
To ease the notation we use only one free sequence (natural number) variable. With the notation of Definition~\ref{tr} $$\forall\xi'' \frt\xi' \; (C(\xi'', \xi') \to (\theta(\xi'')\leftrightarrow \tau(\theta)(\xi'))$$
\end{prop}

\begin{proof}
By formula induction on the complexity of $\theta$. The interesting cases are the second-order quantifier cases. The first order atomic case is immediate, the case for 
$\tau(\xi_i''(t_1)=t_2)$ follows from  the definitions. The first order quantifier cases can be handled similarly as in \cite{eszmmlq2}.
\medskip

For the universal quantifier case
 first assume that $\forall \xi''\psi_1$ and $CODE(\xi')$. By Proposition~\ref{cd2} there is a unique $\eta$ such that $C(\eta, \xi')$ holds and from $\forall \xi''\psi_1$ 
we also have $\psi_1(\eta)$. Then by the induction hypothesis $\tau(\psi_1)(\xi')$, so  $\frt \xi'\;(CODE(\xi') \to \tau(\psi_1))$, ie. $\tau(\forall \xi''\psi_1)$ holds.
\smallskip

Next assume $\tau(\forall \xi''\psi_1)$, ie.  $\frt \xi'\;(CODE(\xi') \to \tau(\psi_1))$, we have to prove $\forall \xi''\;\psi_1$. By Proposition~\ref{cd2}  for any $\xi''$ there is a unique $\xi'$ such that $C(\xi'',\xi')$ and since from this $CODE(\xi')$ follows, we have $\tau(\psi_1)(\xi')$ by our assumption. We can apply the induction hypothesis to get $\psi_1(\xi'')$.

\medskip

For the existential quantifier case
 first assume that $\exists \xi''\psi_1$. For such $\xi''$, by Proposition~\ref{cd2} there is $\xi'$ with $C(\xi'', \xi')$, so $CODE(\xi')$ holds. By the induction hypothesis $\tau(\psi_1)(\xi')$ and we have $\ert \xi' \;(CODE(\xi') \wedge \tau(\psi_1)(\xi'))$ ie. $\tau(\exists \xi''\psi_1)$.

\medskip
Finally assume $\tau(\exists \xi''\psi_1)$, ie. $\ert \xi' \;(CODE(\xi') \wedge \tau(\psi_1)(\xi'))$.  By Proposition~\ref{cd2} for a witness $\xi'$ there is (a unique) $\eta$ such that $C(\eta,\xi')$ holds. By the induction hypothesis  $\psi_1(\eta)$ and we are done.
\end{proof}

The  next theorem is immediate.

\begin{thm}
For an $L_1$-sentence $\theta$, $\tau(\theta)$ is an equivalent (in the axiom system $\mathscr{T}$ from \cite{sco1} with R-KS added) $LOR$-sentence recursively obtained from $\theta$ .
\end{thm}\qed

\section{Conclusion}
In this paper we gave an axiomatic background to the high complexity of intuitionistic real algebra; the fact that second order arithmetic is interpretable in these structures follows from intuitionistically plausible axioms. For an argument in favor of the schema R-KS see the Appendix of \cite{eszmjsl} where its consistency with $\mathscr T$ is also shown. 

Further investigations might address the following questions.
\be
\item Can the axiom system (R-KS in particular) be weakened but still result in  the same level of complexity?
\item Does R-KS have other noteworthy consequences?
\item What (if any) intuitionistically plausible strengthening of R-KS is consistent with $\mathscr T$?
\item In $\S 3.$ of \cite{eszmjsl} we showed that intuitionistic second-order structures built on a certain type of complete Heyting Algebra (called 'nice' in that paper) are models of R-KS. A question arises can this sufficient condition be weakened? Can we axiomatize properties of the second-order structure on the level of the underlying Heyting Algebra? 
\ee

\end{document}